\documentclass[leqno,a4paper,12pt]{amsart}
\usepackage{amsmath,amsthm,amssymb,amsfonts,enumerate,color}
\allowdisplaybreaks[4]
\numberwithin{equation}{section}
\newtheorem{thm}{Theorem}[section]
\newtheorem{prop}[thm]{Proposition}
\newtheorem{lem}[thm]{Lemma}
\newtheorem{remark}[thm]{Remark}
\newtheorem{defn}[thm]{Definition}

\newcommand{\real}{{\mathbb R}}

\newcommand{\norm}[1]{\left\Vert#1\right\Vert}
\newcommand{\abs}[1]{\left\vert#1\right\vert}

\renewcommand{\L}{{\mathcal L}}

\begin{document}

\title[Lipschitz spaces related to biharmonic operator]
{Regularity property of some operators on  Lipschitz spaces and homogeneous Lipschitz spaces  related to biharmonic operator}

\author{Chao Zhang }

 \address{School of Statistics and Mathematics \\
             Zhejiang Gongshang University \\
             Hangzhou 310018, People's Republic of China}
 \email{zaoyangzhangchao@163.com}

\thanks{Supported by the National Natural Science Foundation of China(Grant Nos. 11971431, 11401525), the Natural Science Foundation of Zhejiang Province(Grant No. LY18A010006), and the first Class Discipline of Zhejiang-A(Zhejiang Gongshang University-Statistics)¡£}
  %\date{\today}

\begin{abstract}
In this paper, we consider the characterizations of the Lipschitz spaces and homogeneous Lipschitz spaces associated to the biharmonic operator $\Delta^2.$ With this characterizations, we prove the boundedness of the Bessel potentials, fractional integrals, fractional powers, Riesz transforms and multipliers of Laplace transforms type associated to $\Delta^2$ on Lipschitz spaces and homogeneous Lipschitz spaces.  The proofs of these results need the language of semigroups in an essential way.
\end{abstract}

\maketitle

 %\tableofcontents
\vskip 0.3cm

\section{Introduction and statement of the main results}

Classical Lipschitz spaces on $\real^n$, $\Lambda^\alpha,$ for $0< \alpha <1$,  are defined as the set of functions
			$\varphi$  such that  $ |\varphi(x+z)-\varphi(x)| \le C |z|^\alpha,$ $x,z\in\real^n$.
These spaces are of particular importance in Harmonic Analysis and PDE's.
It goes  back to A. Zygmund that for $\alpha =1$ the natural space is the set of functions such that
$ |\varphi(x+z)+\varphi(x-z)-2\varphi(x)  | \le C |z|$, $x,z\in\real^n$.
 For $\alpha>1$, $\Lambda^\alpha$ is defined as the class of smooth functions such that their first order derivatives belong to $\Lambda^{\alpha-1}.$
 Given the pointwise characterization of the above definitions,   the pointwise expression of the operator is needed to prove the boundedness of operators in the classes $\Lambda^\alpha$. However, it is relatively common to have definitions of operators in terms of the heat or the Poisson semigroup, for example, the Bessel potential, negative powers, fractional powers, and so on. Then, it should be desirable to have some equivalent definitions of the Lipschitz classes, which contains some expressions of the semigroups(not on points). As far as we know,  the first attempt  in this direction goes back to the Ph.D thesis of M. Taibleson,  see \cite{TaiblesonThesis}.
He characterized the   bounded Lipschitz functions via the Poisson semigroup, $e^{-t\sqrt{-\Delta}},$ and  the Gauss semigroup, $e^{t{\Delta}}$,  see \cite{Stein, Taibleson} also. The advantage of this approach is that, the semigroup language allows us to obtain regularity results in these spaces in a more direct way. In particular, it allows us to prove the boundedness of some fractional operators, such as fractional laplacians, fractional integrals, Riesz transforms and Bessel potentials, in a much more simple way than using the classical definition of the Lipschitz spaces.

Recently, some papers have been devoted to analyze the
 Lipschitz spaces adapted to different  ``laplacians'' and to find  pointwise
		and semigroup estimate characterizations, see \cite{ MartaT, dLC-T, Urbina, Sjogren, ST2}.  In the case of the Ornstein-Ulhenbeck operator $\mathcal{O}=-\frac{1}{2}\Delta+x\cdot\nabla$, in \cite{Urbina}, some Lipschitz classes were defined by means of its Poisson semigroup, $e^{-t\sqrt{\mathcal{O}}}$, and in \cite{Sjogren} a pointwise characterization was obtained for $0<\alpha<1$. In the literature sometimes  ``Lipschitz classes'' are also known as ``H\"older classes''.   In the case of the Hermite operator $\mathcal{H}=-\Delta+|x|^2$, adapted H\"older classes were defined pointwisely in \cite{ST2}.    By using semigroups these last classes were characterized in \cite{dLC-T}, also in the parabolic case.   The classical parabolic case was treated in   \cite{ST3}. In \cite{MartaT}, the authors proved the characterization of Lipschitz spaces adapted to the Schr\"odinger operators $-\Delta+V$, where $V$ is a nonnegative potential satisfying a reverse H\"older inequality.
		
		In this paper, we shall define the Lipschitz space  $\Lambda_{\alpha/4}^{\Delta^2}$,  associated to the biharmonic operator  $\Delta^2$, see Definition \ref{Lipdef1}. This  definition will allow us to prove some
regularity estimates for operators defined through the heat semigroups generated by $\Delta^2$, see Theorem \ref{thm:charclassic}.

All the above semigroup characterizations are given for special sets of Lipschitz functions, either  bounded, or polynomially bounded, see \cite{MartaT, Stein}. However this is not enough in our case. Even more, it is also not enough in the classical case. For example, it is well known that the Hilbert transform doesn't map bounded functions into bounded functions, see \cite{Duo, Grafakos}. And  the result developed  by Privalov  establishes that, if a function $f$ satisfies
\begin{equation}\label{Lipschitzcon}
\abs{f(x)-f(y)}\le C|x-y|^\alpha,
\end{equation}
then the conjugate function $\tilde f$ satisfies
\begin{equation*}
\abs{\tilde f(x)- \tilde f(y)}\le C|x-y|^\alpha,
\end{equation*}
for $0<\alpha<1,$ see \cite{Torchinsky, Zygmund}.
In other words if we define the norm

$$
\|f\|_{Lip^\alpha} = \sup_{x\neq y} \frac{\abs{f(x)-f(y)}}{|x-y|},
$$
then $$ \|\tilde{f}\|_{Lip^\alpha} \le C \| f\|_{Lip^\alpha}.
$$
But the conjugate operator $T:f\mapsto \tilde{f}$ is not well defined for functions on $L^\infty(\real^n)$.
The above situation can be also found in other scales of spaces in Harmonic Analysis, such as Besov spaces, Triebel-Lizorkin spaces and so on. The standard adjective used   for this type spaces  defined above is ``homogeneous" and the standard notation is to put a ``dot " over the ``non homogenous space". Motivated by these ideas, we will   consider the homogenous Lipschitz spaces, ${\dot\Lambda_{\alpha/4}}^{\Delta^2}$, see Definition \ref{defclasica}.  This will allow us to prove some regularity results for operators defined spectrally in $L^2(\real^n)$, but not defined in $L^\infty(\real^n)$  in general, see Theorems \ref{negativepower}-\ref{multiplicador}.

%\subsection{Some estimate on semigroups}
Consider the following Cauchy problem for the biharmonic heat equation
\begin{equation}\label{el.4}
\begin{cases}
(\partial_t+\Delta^2)u(x,t)=0 \ \ \  \text{in } \  \mathbb R_+^{n+1}\\
u(x,0)=f(x) \ \ \  \ \ \  \ \ \  \  \  \text{in } \  \mathbb R^{n}.
\end{cases}
\end{equation}
Its solution is  given by
$$u(x,t)=W_tf(x)=\int_{\real^n}W_t(x-y)f(y)dy,$$
where
$$W_t(x)=\mathfrak{F}^{-1}\left(e^{-\abs{\xi}^4t}\right)=t^{-{n/ 4}}g\left({x \over {t^{1/  4}}}\right)$$
 with
 \begin{equation*}\label{equg}
 g(\xi)=(2\pi)^{-{n/2}}\int_{\real^n}e^{i\xi \eta-\abs{\eta}^4}d\eta=\alpha_n\abs{\xi}^{1-n}\int_0^\infty e^{-s^4}(\abs{\xi} s)^{n/2}J_{(n-2)/2}(\abs{\xi} s)ds,\ \ \ \xi\in \real^n,
 \end{equation*}
 and $\mathfrak{F}^{-1}$ being the inverse Fourier transform.
 Here, $J_v$ denotes the $v$-th Bessel function and $\alpha_n>0$ is a normalization constant such that
 \begin{equation*}%\label{equ:integralConst}
 \int_{\real^n}g(\xi)d\xi=1.
 \end{equation*}
 See \cite{KochL, WangCY} for more details.
Then, we have the following several results by classical analysis(for details, see \cite{SW}):
 \begin{itemize}
\item[(1)]
 If $f\in L^p(\real^n)$, $1\le p\le \infty,$ then $$\displaystyle \lim_{t\rightarrow 0}u(x,t)=f(x)\quad  a.e. \quad  x\in \real^n,$$ and
\begin{equation*}
\norm{u(\cdot, t)}_{L^p(\real^n)}\le C\norm{f}_{L^p(\real^n)}.% \quad 1\le p\le \infty.
\end{equation*}
\item [(2)] If $1\le p<\infty,$ then
\begin{equation*}
\norm{u(\cdot, t)-f}_{L^p(\real^n)}\rightarrow 0, \quad  \hbox{when} \quad t\rightarrow 0^+.
\end{equation*}
\end{itemize}

Now, we define  the  spaces $\Lambda^{\Delta^2}_{\alpha/4}$  in the following.

\begin{defn}\label{Lipdef1} Let $\alpha >0.$ We define the spaces $\Lambda^{\Delta^2}_{\alpha/4}$  as
\begin{align*}\Lambda_{\alpha/4}^{\Delta^2}=&\Big\{ f \in  L^\infty(\real^n)
	:\;
	 \left\|\partial_t^k	{W}_t f \right\|_{L^\infty(\mathbb{R}^{n})}\leq C_\alpha t^{-k+\alpha/4},\:  k=[\alpha/4]+1  \Big\}.\end{align*}
	We  endow this space with the norm $$\|f\|_{\Lambda_{\alpha/4}^{\Delta^2}} := \norm{f}_{\infty} + {S}_\alpha[f],$$
where ${S}_\alpha[f]$  is the infimum of the constants $C_\alpha$ appearing above.
\end{defn}

And, we have the following characterization theorem.

 \begin{thm}\label{thm:charclassic}
Let $0<\alpha<2$. Then the following three statements are equivalent:
\begin{itemize}
	\item [(1)] $f\in \Lambda_{\alpha/4}^{\Delta^2}$.
\item [(2)]
$f\in \Big\{ f \in  L^\infty(\real^n)
	:	 \left\|\partial_t^k	{e}^{-t\sqrt{-\Delta} }f \right\|_{L^\infty(\mathbb{R}^{n})}\leq C_\alpha t^{-k+\alpha},\:  k=[\alpha]+1  \Big\}$,
where ${e}^{-t\sqrt{-\Delta} }$ is the classical Poisson kernel.
	%$$\|f\|_{\Lambda_{\alpha}} := \norm{f}_{\infty} + \tilde{S}_\alpha [f], $$
%where $\tilde{S}_\alpha[f]$  is the infimum of the constants $C_\alpha$ appearing above.
\item [(3)] $\displaystyle f\in \Big\{ f \in  L^\infty(\real^n)
	:	 N_\alpha[f]:= \sup_{|y|>0}\frac{\|f(\cdot+y)+f(\cdot-y)-2f(\cdot)\|_\infty}{|y|^\alpha}<\infty\}$.
\end{itemize}
Moreover, $$\norm{f}_{\Lambda_{\alpha/4}^{\Delta^2}}\sim \norm{f}_{L^\infty}+\tilde S_\alpha[f]\sim \norm{f}_{L^\infty}+ N_\alpha[f],$$ where $\tilde S_\alpha[f]$ denotes the infimum of the constants $C_\alpha$  appearing in $(2)$.
 \end{thm}

We shall prove the regularity property of the Bessel potential of order $\beta>0$  associated to $\Delta^2$.  Its definition is motivated by the Gamma formula, $$
	 (Id+\Delta^2)^{-\beta/4}f(x)=\frac{1}{\Gamma(\beta/4)}\int_0^\infty e^{-s}e^{-s\Delta^2}f(x) s^{\beta/4}\frac{ds}{s},
	  $$
see  \cite{ST}.
In fact, we have the following theorem.
	  \begin{thm}\label{Schau}
	Let $\alpha,\beta>0$. Then, $(Id+\Delta^2)^{-\beta/4}$  satisfies
\begin{itemize}
\item[(i)]
		$\|(Id+\Delta^2)^{-\beta/4} f\|_{\Lambda_{\frac{\alpha+\beta}{4}}^{\Delta^2}}\le C \|f\|_{	 \Lambda_{\alpha/4}^{\Delta^2}},$
\item[(ii)] $\|(Id+\Delta^2)^{-\beta/4} f\|_{\Lambda_{{\beta/4}}^{\Delta^2}}\le C \|f\|_{\infty}.$		
		\end{itemize}
	  \end{thm}
	\vskip 0.3cm

As we have said previously, operators defined by using the heat semigroups are well defined in $L^2(\real^n)$, but not in $L^\infty(\real^n)$, unless the heat kernel has good enough property.  So, it seems natural to study the following type Lipschitz space.

\begin{defn}[Homogeneous Lipschitz spaces associated to $\Delta^2$]\label{defclasica} Let $\alpha >0.$ We define the semi-norm spaces $\dot{\Lambda}_{\alpha/4}^{{\Delta^2}}$  as
\begin{align*}\dot{\Lambda}_{\alpha/4}^{{\Delta^2}}=&\Big\{f
	:\;
	 \left\|\partial_t^k	{W}_t f \right\|_{L^\infty(\mathbb{R}^{n})}\leq C_\alpha t^{-k+\alpha/4},\:  k=[\alpha/4]+1  \Big\}.\end{align*}
	We  endow this class with the semi-norm $$\|f\|_{\dot{\Lambda}_{\alpha/4}^{{\Delta^2}}} = {S}_\alpha[f].$$
\end{defn}

Of course, in the definition of the space $\dot{\Lambda}_{\alpha/4}^{{\Delta^2}}$, we should assume that $\partial_t^k W_tf$ is well defined, but not necessarily $f\in L^\infty(\real^n)$. However, for certain classes of functions, these spaces $\dot{\Lambda}_{\alpha/4}^{{\Delta^2}}$ contain the functions which satisfy a pointwise inequality of the type \eqref{Lipschitzcon}. In fact, we have the following theorem.

\begin{thm}\label{nuevoStein}
				Let $0<\alpha < 2$,  and let $f$ be functions such that $\partial_t W_t f$ is well defined and $\displaystyle \lim_{t\rightarrow 0^+}W_tf(x)=f(x)$ $a.e. \ x\in \real^n$.  Then the following two statements are equivalent:
\begin{itemize}
	\item [(1)] $f\in \dot{\Lambda}_{\alpha/4}^{{\Delta^2}}$.
    \item [(2)] $\displaystyle f\in \{f: N_\alpha[f]  <\infty \}$.
\end{itemize}
				Moreover, $$\|f \|_{\dot{\Lambda}_{\alpha/4}^{{\Delta^2}} }  \sim N_\alpha[f]. $$
			\end{thm}

By the  theorem above and Proposition 3.28 in \cite{MartaT}, we know that, when		
		 $0<\alpha < 2$, if  $f \in \dot \Lambda^{{\Delta^2}}_{\alpha/4}$, then $\displaystyle \sup_{|y|>0} \frac{\|f(\cdot-y)-f(\cdot) \|_\infty}{|y|^\alpha} < \infty.$

Also, we will prove a parallel result to \cite[Proposition 9 in pp. 147]{Stein}, which states the relationship between the Lipschitz function and its derivatives.
\begin{thm}\label{thm:LipDirivative}
Suppose that $\alpha>1$.  Then,
$$f\in \dot \Lambda^{\Delta^2}_{\alpha/4}\quad \hbox{if, and only if},\quad \partial_{x_i}f\in \dot \Lambda^{\Delta^2}_{(\alpha-1)/4},\ i=1,\cdots, n. $$
In this case, the following equivalence holds
$$\norm{f}_{\dot \Lambda^{\Delta^2}_{\alpha/4}}\sim \sum_{i=1}^{n}\norm{\partial_{x_i}f}_{\dot \Lambda^{\Delta^2}_{(\alpha-1)/4}}.$$
\end{thm}
With the spaces $\dot{\Lambda}_{\alpha/4}^{{\Delta^2}}$ defined in Definition \ref{defclasica}, we can get the boundedness of some operators associated to $\Delta^2$  in the homogeneous  Lipschitz spaces. The operators we will consider are defined as follows.

\begin{itemize}
\item The {\it fractional integral} of order $\beta>0$,
	  $$
(\Delta^2)^{-\beta/4}f(x)=\frac{1}{\Gamma(\beta/4)}\int_0^\infty e^{-s\Delta^2}f(x) s^{\beta/4}\frac{ds}{s}.
	  $$
\item The {\it fractional biharmonic operator} of order $\beta>0$,
	   $$
(\Delta^2)^{\beta/4}f(x)= \frac1{c_\beta} \int_0^\infty
	   (e^{- s\Delta^2}-Id)^{[\beta/4]+1}f(x)\,\frac{ds}{s^{1+\beta/4}}.
	   $$	
	   {\item The first order Riesz transforms  defined by
	  $$\mathcal{R}_i=\partial_{x_i}(\Delta^2)^{-1/4}
\hbox{  and } \,  R_i = (\Delta^2)^{-1/4}\partial_{x_i},  \: i=1,\cdots,n. $$}
\end{itemize}

\begin{thm}\label{negativepower}
	Let $\alpha,\beta>0$. Then, $(\Delta^2)^{-\beta/4}$  satisfies
		$$\|(\Delta^2)^{-\beta/4} f\|_{\dot{\Lambda}_{\frac{\alpha+\beta}{4}}^{\Delta^2}}\le C \|f\|_{\dot{\Lambda}_{\alpha/4}^{{\Delta^2}}}.$$
	  \end{thm}
\begin{thm}[H\"older estimates]\label{Holderestimates}
	  	Let $0< \beta < \alpha $ and  $f\in	 \dot{\Lambda}_{\alpha/4}^{{\Delta^2}}.$  Then,
	  $$	\| (\Delta^2)^{\beta/4} f\|_{\dot \Lambda_{{\frac{\alpha-\beta}{4}}}^{\Delta^2}}\le C \|f\|_{\dot{\Lambda}_{\alpha/4}^{{\Delta^2}}}.
	  	$$
	  \end{thm}
	  	%For the classes $\Lambda^\alpha_\L,$ $0< \alpha < 1,$   the result was obtained in \cite{MSTZ}.
		\vskip 0.3cm
	  	  \begin{thm} \textcolor{white}{} \label{TRiesz}
	  	\begin{itemize}
	  		\item [(1)] For {$0<\alpha\le 1$,}  then $\|\mathcal{R}_if\|_{\dot{\Lambda}_{\alpha/4}^{{\Delta^2}}} \le C \|f\|_{\dot{\Lambda}_{\alpha/4}^{{\Delta^2}}} $,  $i=1,\dots,n$.
		\item [(2)] For  {$1<\alpha\le 2$,}  then $\|{R}_if\|_{\dot{\Lambda}_{\alpha/4}^{{\Delta^2}}} \le C \|f\|_{\dot{\Lambda}_{\alpha/4}^{{\Delta^2}}} $,  $i=1,\dots,n$.
	  		\end{itemize}	
	  		  \end{thm}

\begin{remark}
If we note that $(\Delta^2)^{-\beta/4}=\Delta^{-\beta/2}$, $(\Delta^2)^{\beta/4}=\Delta^{\beta/2}$, $\mathcal{R}_i=\partial_{x_i}(\Delta^2)^{-1/4}=\partial_{x_i}\Delta^{-1/2}$ and $R_i = (\Delta^2)^{-1/4}\partial_{x_i}$ for $i=1,\cdots, n,$ then  Theorems \ref{negativepower}--\ref{TRiesz} reveal  the regularity of the above  operators associated to $\Delta$ in the homogeneous Lipschitz spaces associated to $\Delta^2.$
\end{remark}

	Also, we can get the following result related to the Laplace transform.
	  \begin{thm}\label{multiplicador}
	  	Let $a$ be a compactly supported bounded measurable function on $[0,\infty)$ and consider
	  	$$
	  	m(\lambda)=\lambda\int_0^\infty e^{-s\lambda}a(s)ds, \, \, \lambda >0.
	  	$$
	 Then, the multiplier operator of the Laplace transform type $m(\Delta^2)$ is bounded in $L^2(\real^n),$ and,  for every $\alpha>0$,  $m(\Delta^2)$ is bounded from  $\dot{\Lambda}_{\alpha/4}^{{\Delta^2}}$ into itself.
	  \end{thm}

The organization of the paper is the following.   In Section \ref{Sec:2}, we prove the characterization of the classical Lipschitz spaces related to $\Delta^2$  and the boundedness of the Bessel potentials on this kind of Lipschitz spaces.  In Section \ref{Sec:4}, we will prove the characterization of the homogeneous Lipschitz spaces related to $\Delta^2$  and the boundedness of the fractional integrals, fractional powers, Reisz transforms and the multiplier of the Laplace transforms on the homogeneous Lipschitz spaces  associated to $\Delta^2$.

Along this paper, we will use the variable constant convention, in which $C$ denotes a
constant that may not be the same in each appearance. The constant will be written with
subindexes if we need to emphasize the dependence on some parameters.

\bigskip
\section{Proof of Theorems \ref{thm:charclassic} and \ref{Schau} }\label{Sec:2}

\setcounter{equation}{0}

In this section, we will give the proof of Theorems \ref{thm:charclassic} and \ref{Schau}. Firstly, we  list some kernel estimations for the heat kernel $W_t(x)$ of the semigroup
$e^{-t\Delta^2}$.

\begin{lem}\rm{(See \cite{Davies} and \cite[Lemma 2.4]{KochL})} \label{le2.2}
For $x \in \real^n$ and $t>0$, the following estimates hold:
 \begin{itemize}
\item[(i)]
\begin{equation*}
\abs{W_t(x)}\leq Ct^{-n/4}e^{{-c \frac{\abs{x}^{4/ 3}}{t^{1/3}}}},\ \  c={3\cdot 2^{1/ 3}\over 16},
\end{equation*}

\item[(ii)]
\begin{equation*}
 \abs{\partial_t^l \nabla^k W_t(x)}\leq C\left( t^{\frac{1}{4}}+\abs{x} \right)^{-n-k-4l}e^{{-c' \frac{\abs{x}^{4/ 3}}{t^{1/ 3}}}},~~~~~ \forall k, l\ge 1, c'<c,
\end{equation*}

\item[(iii)]
\begin{equation*}
 \norm{\partial_t^l \nabla^k W_t(\cdot)}_{L^1(\real^n)}\leq Ct^{-l-{k\over 4}},~~~~~ \forall k, l\ge 1,
\end{equation*}

\item[(iv)] there exist $C, C'>0$ such that for $0\le j\le 4,$
\begin{equation*}
 \abs{\nabla^j W_t(x)} \leq Ce^{-C'\abs{x}},~ ~ ~ ~ ~ \forall (x,t)\in \real^n\times (0,\ 1)\backslash\Big(B_2\times(0, {1\over 2})\Big).
\end{equation*}
 \end{itemize}
\end{lem}

%\subsection{ Lipschitz spaces associated to $(-\Delta)^2$}\label{sec:lips}

\begin{remark}
We observe that, because of the estimate {\rm(iii)} in Lemma \ref{le2.2}, we can assume  $t<1$ in Definition \ref{Lipdef1}.
\end{remark}
The following proposition shows that we can use any bigger integer than $[\alpha/4]+1$ in Definition \ref{Lipdef1}.
\begin{prop}\label{subiraK}
		Let {$\alpha>0$}. A function   $f\in \Lambda_{\alpha/4}^{\Delta^2}$ if, and only if,   for   $m \ge [\alpha/4]+1$, we have $\| \partial_ t^{m} {W}_t f\|_ {L^\infty(\mathbb{R}^n)} \le C_m t^{-m+ \alpha/4}$ and $f\in L^\infty(\real^n).$
	\end{prop}
	\begin{proof}  Let $m \ge  [\alpha/4]+1=k$.
By Lemma  \ref{le2.2}, we get $\displaystyle \abs{\int_{\mathbb{R}^n} \partial_t^k W_t(x,y) dy} \le \frac{C}{t^k}.$ Then, by the semigroup property we have
		\begin{eqnarray*}
 \Big| \partial_t^{m}  W_t f(x) \Big|=  C \Big| \partial_t^{m-k}W_{t/2}\left(\partial_u^{k}W_{u}f(x)\big|_{u=t/2}\right) \Big|
			\le  C_\alpha'\frac1{t^{m-k}} t^{-k +\alpha/4} = C_m t^{-m+\alpha/4}.
		\end{eqnarray*}
		 For the converse, the fact    $| \partial_t^{m} W_t f(x) | \to 0 \text{ as } t\to\infty,$ see (ii) in Lemma  \ref{le2.2}, allows us to integrate  on $t$ as many times as we need to  get $\| \partial_t^{k}W_t f\|_ {L^\infty(\mathbb{R}^n)} \le C_\alpha\, t^{-k+ \alpha/4}.$		
	\end{proof}

% \begin{remark}
%In \cite{Stein}, E. M. Stein  proved that $(2)$ and $(3) $ are equivalent.
 %\end{remark}

 Now, we can prove our first main theorem.
We should remember the following Bochner's subordination formula
\begin{equation*}
e^{-t\sqrt{\lambda} } = \frac{1}{2\sqrt{\pi}}\int_0^\infty\frac{te^{-\frac{t^2}{4\tau}}}{\tau^{3/2}} e^{-\tau \lambda }d\tau
=\frac{1}{\sqrt{\pi}}\int_0^\infty\frac{e^{-\tau}}{\tau^{1/2}} e^{-\frac{t^2}{4\tau} }d\tau, \quad \lambda >0.
\end{equation*}
Given a positive operator $L$, the above formula  gives a good representation of the Poisson semigroup associated to the operator $L$. In general, the Poisson semigroup is written as  $e^{-t\sqrt{L}}$. This procedure could be iterated, and then we can get the semigroup
$e^{-t \sqrt{\sqrt{L}}}$.  In the case of $\Delta^2$, it seems natural that
$e^{-t \sqrt{\sqrt{\Delta^2}}}$ gives the classical Poisson semigroup. In the proof of Theorem \ref{thm:charclassic}, we shall use this idea in a fundamental way.

 \begin{proof}[Proof of Theorem \ref{thm:charclassic}]
 Since $0<\alpha<2,$ we have $[\alpha/4]+1=[\alpha/2]+1=1$.

 ``$(1)\Rightarrow (2).$" Let  $f\in \Lambda_{\alpha/4}^{\Delta^2}$.    By \cite[Lemma 5 in pp. 145]{Stein}, it is enough to prove that $$\|\partial_t^{2}{e}^{-t\sqrt{-\Delta} } f\|_\infty\le Ct^{-2+\alpha}.$$
Since $\displaystyle \partial_t^2\left( \frac{te^{-\frac{t^2}{4\tau}}}{\tau^{3/2}}\right)=\partial_\tau \left(\frac{t e^{-\frac{t^2}{4\tau}}}{\tau^{3/2}}\right)$,  the Bochner's subordination formula and integration by parts give
			\begin{align*}
			&|\partial_t^{2}e^{-t\sqrt {-\Delta}} f(x)|=\left| \frac{1}{2{\pi}}\int_0^\infty\partial_t^{2}\left(\frac{te^{-\frac{t^2}{4\tau}}}{\tau^{3/2}}\right) \int_0^\infty \frac{e^{-u}}{u^{1/2}}W_{\frac{\tau^2}{4u}}f(x)dud\tau\right|\\
&=	\left| \frac{1}{2{\pi}}\int_0^\infty\partial_\tau\left(\frac{te^{-\frac{t^2}{4\tau}}}{\tau^{3/2}}\right) \int_0^\infty \frac{e^{-u}}{u^{1/2}}W_{\frac{\tau^2}{4u}}f(x)dud\tau\right|\\
&=\left| \frac{1}{2{\pi}}\int_0^\infty  \frac{te^{-\frac{t^2}{4\tau}}}{\tau^{3/2}}  \int_0^\infty \frac{e^{-u}}{u^{1/2}}\partial_s W_{s}f(x)\big|_{s=\frac{\tau^2}{4u}}\ {\tau\over 2u}dud\tau\right|\\
& \le C \,S_\alpha [f] \int_0^\infty \frac{t e^{-\frac{t^2}{4\tau}}}{\tau^{3/2}}\int_0^\infty \frac{e^{-u}}{u^{1/2}}{\tau\over 2u} \left({\tau^2\over 4u}\right)^{-1+\alpha/4}dud\tau\\
			&\le C\,S_\alpha[f] t^{-2+\alpha}.
			\end{align*}

``$(2)\Rightarrow (3).$" It is already proved in  \cite[Lemma 5 and Proposition 8 in Chapter V]{Stein}.

``$(3)\Rightarrow (1).$" Since  $$\displaystyle  \int_{\mathbb{R}^{n}} \partial_y {W}_t(y)f(x+y) dy=\int_{\mathbb{R}^{n}} \partial_t {W}_t(-y)f(x-y)dy=\int_{\mathbb{R}^{n}} \partial_t {W}_t(y)f(x-y)dy$$
				and  $$\displaystyle \int_{\real^n}\partial_t {W}_t(y)dy=0,$$ we have
				\begin{align*}
				\abs{\partial_t {W}_tf(x)} &= \left|\frac{1}{2}\int_{\mathbb{R}^{n}}\partial_t {W}_t(y)(f(x-y)+f(x+y)-2f(x)) dy\right| \\&\le C N_\alpha[f]\,  \int_{\mathbb{R}^n}\frac{e^{-c\frac{|y|^{4/3}}{t^{1/3}}}}{t^{{n}/{4}+1}}|y|^\alpha dy\le C N_\alpha[f]\,   t^{-1+\alpha/4}.
				\end{align*}
We complete the proof of Theorem \ref{thm:charclassic}.		
 \end{proof}

\begin{remark}\label{rem:all}
We should note that the proof of ``$(1)\Rightarrow (2)$" in the above theorem is valid for any $\alpha>0.$
\end{remark}

Now, we are in a position to prove Theorem \ref{Schau}.

	 \begin{proof} [{\it Proof of Theorem \ref{Schau}}] We prove only (i), estimate (ii) can be proved analogously.
Firstly, since $f\in L^\infty(\real^n)$,  we get
\begin{align*}
	 & \abs{(Id+\Delta^2)^{-\beta/4} f(x)}=\Big| \frac{1}{\Gamma(\beta/4)}\int_0^\infty e^{-s}W_s f(x) s^{\beta/4}\frac{ds}{s}\Big|\\
&\le C\, \|f\|_{L^\infty(\real^n)}\,   \int_0^{\infty} e^{-s} s^{\beta/4}\frac{ds}{s}\le C\, \|f\|_{L^\infty(\real^n)}, \: \  \forall  x\in\real^n.
	  \end{align*}

Secondly,   Fubini's Theorem allows us to get $$\displaystyle W_t((Id+{\Delta^2})^{-\beta/4}f)(x)=
	 \frac{1}{\Gamma(\beta/4)}\int_0^\infty e^{-s} W_t(W_s f)(x) s^{\beta/4}\frac{ds}{s}.$$
For  $j=[\alpha/4]+1$ , by the semigroup property,  we have
	 \begin{multline*}
 \int_0^\infty \Big|e^{-s} \partial^j_t W_t(W_s f)(x) \Big|s^{\beta/4}\frac{ds}{s}
	 =  \int_0^\infty \Big| e^{-s} \partial^j_w W_w f(x)\big|_{w=t+s} \Big|s^{\beta/4}\frac{ds}{s}\\
 \le C S_\alpha[f] \int_0^\infty e^{-s}  (t+s)^{-j+\alpha/4} s^{\beta/4}\frac{ds}{s}.
	 \end{multline*}
{The function in the last integral can be bounded by a uniform (in a neighborhood of $t$) integrable function  (of $s$).}  This means that  we can interchange the derivative with respect to $t$ and the integral with respect to $s$ in the above expression.
	
	  Let $\ell=[\alpha/4+\beta/4]+1$.  By iterating  the above arguments and using the hypothesis, we have
	  	\begin{align*}
	  	&|\partial_t^\ell W_t ((Id+{\Delta^2})^{-\beta/4} f)(x)|=\left|\frac{1}{\Gamma(\beta/4)}\int_0^\infty e^{-s}\partial_t^\ell W_t(W_s f)(x) s^{\beta/4}\frac{ds}{s}\right|\\
	  	&\le {C}\int_0^\infty e^{-s}\abs{\partial_w^\ell W_w f(x)\big|_{w=t+s}}s^{\beta/4}\frac{ds}{s}
		\\&\le {C \ S_\alpha[f]\,}\int_0^\infty e^{-s}(t+s)^{-\ell+\alpha/4}s^{\beta/4}\frac{ds}{s}
	  	\\&\stackrel{\frac{s}{t}=u}{\le }C\, S_\alpha[f]\,{t^{\alpha/4+\beta/4-\ell}}\int_0^\infty \frac{u^{\beta/4}e^{-tu}}{(1+u)^{\ell-\alpha/4}}\frac{du}{u}
	  	\\&\le C\ S_\alpha[f] \, t^{\alpha/4+\beta/4-\ell}.
	  	\end{align*}		
We end the proof of the theorem.
	  \end{proof}

\vskip 0.8cm
	  \section{Proof of Theorems \ref{nuevoStein}--\ref{multiplicador} }\label{Sec:4}

In this section, we will give the proof of Theorems \ref{nuevoStein}--\ref{multiplicador} related with the homogeneous Lipschitz spaces associated to $\Delta^2$.

First, following the proof of Proposition \ref{subiraK}, we can get a  proposition as follows.

\begin{prop}\label{subirk}
		Let {$\alpha>0$}. A function   $f\in \dot \Lambda_{\alpha/4}^{{\Delta^2}}$ if, and only if,   for  all $m \ge [\alpha/4]+1$, we have $\| \partial_ t^{m} {W}_t f\|_ {L^\infty(\mathbb{R}^n)} \le C_m t^{-m+ \alpha/4}$.
	\end{prop}

And, we give a lemma which will be used later.

\begin{lem}\label{derivX}
		Let $\alpha>0$ and $k=[\alpha/4]+1$. If $ \left\|\partial_t^k	{W}_t f \right\|_{L^\infty(\mathbb{R}^{n})}\leq C_\alpha t^{-k+\alpha/4}$,  then for every $j,m\in\mathbb N$ such that ${m}/{4}+j\ge k$,  there exists a constant $C_{m,j}>0$ such that
		$$
		\left\| {\partial_{x_i}^m\partial_t^j {W}_tf}\right\|_{\infty}\le C_{m,j,\alpha}	t^{-(m/4+j)+\alpha/4}, \text{ for every } i=1\dots,n.
		$$
		%Moreover,  for each ${j,m},$ the constant  $C_{m,j}$ is comparable to the constant $C_\alpha$ in Definition \ref{Lipdef1}.	
		\end{lem}
		
	\begin{proof}If $j\ge k$, by the semigroup property we get
		\begin{align*}
		\left|\partial_{x_i}^m\partial_t^j{W}_tf(x)\right|&=C  \left|\int_{\real^n} \partial_{x_i}^m\partial_v^{j-k}{W}_{v}(x-z)\big|_{v=t/2}\partial_u^k {W}_uf(z)\big|_{u=t/2}dz\right|\\
		&\le  \frac{C_{m,j,\alpha} \|\partial_u^k {W}_uf\big|_{u=t/2}\|_\infty}{t^{{m}/{4}+j-k}} \int_{\real^n}\frac{e^{-c\frac{|x-y|^{4/3}}{t^{1/3}}}}{t^{n/4}}dy\\
		&\le C_{m,j,\alpha}	t^{-({m}/{4}+j)+\alpha/4}, \:\;\quad  x\in\real^n.
		\end{align*}
 If $j<k$, by proceeding as before we get  $	 \left|\partial_{x_i}^m\partial_t^k {W}_tf(x)\right|\le  C	 t^{-({m}/{4}+k)+\alpha/4}$, $ x\in\real^n$, and  we  get the result by integrating the previous estimate $k-j$ times, since  $|\partial_{x_i}^m\partial_t^\ell{W}_tf(x)|\to 0$ as $t\to \infty$ as far as ${m}/{4}+\ell \ge k$. 	
 	\end{proof}

Then, we can prove the  pointwise characterization theorem.

		\begin{proof}[Proof of Theorem \ref{nuevoStein}]
			Let $x\in\real^n$ and $f\in \dot\Lambda_{\alpha/4}^{{\Delta^2}}$. We can write, for every $t>0$, $y\in\real^n$,
				\begin{align*}
				|f(x+y)&+f(x-y)-2f(x)|\le |{W}_tf(x+y)-f(x+y)|+|{W}_tf(x-y)-f(x-y)|\\&\quad \quad+2|{W}_tf(x)-f(x)|+ |{W}_tf(x+y)-{W}_tf(x) +{W}_tf(x-y)-{W}_tf(x)|.
				\end{align*}
We have
				$$
				|{W}_tf(x)-f(x)|=\left|\int_0^t \partial_u {W}_uf(x)du \right|\le C{S}_\alpha[f]\int_0^t u^{-1+\alpha/4}du =C {S}_\alpha[f] t^{\alpha/4}.$$
		In a parallel way, we can handle the two first	summands.  Regarding the last summand,  by using the chain rule, Lemmas \ref{le2.2} and \ref{derivX} we have \begin{align*}
&\abs{{W}_tf(x+y)-{W}_tf(x) +{W}_tf(x-y)-{W}_tf(x)}\\
&=\left|\int_{0}^1\partial_{\theta}\left({W}_tf(x+\theta y)+{W}_tf(x-\theta y)\right) d\theta  \right|\\
&=\left|\int_0^1 \left(\nabla_u {W}_tf(u){\big|_{u=x+\theta y}}\cdot y-\nabla_v {W}_tf(v){\big|_{v=x-\theta y}}\cdot y\right)d\theta\right|\\
&=  \left|\int_0^1\int_{-1}^1\partial_\lambda\nabla_u {W}_tf(u){\big|_{u=x+\lambda\theta y}}\cdot y\,d\lambda d\theta\right|\\&=\left|\int_0^1\int_{-1}^1\nabla_u^2 {W}_tf(u){\big|_{u=x+\lambda\theta y}}\cdot\theta |y|^2\,d\lambda d\theta\right|\\&\le C{S}_\alpha[f] \, t^{-1/2+\alpha/4}|y|^2. \:\: \end{align*}
				 	Thus, by choosing $t=|y|^4$ we get what we wanted.
					
					For the converse, we assume that $ N_\alpha[f]<\infty$. Since  $$\displaystyle  \int_{\mathbb{R}^{n}} \partial_y {W}_t(y)f(x+y) dy=\int_{\mathbb{R}^{n}} \partial_t {W}_t(-y)f(x-y)dy=\int_{\mathbb{R}^{n}} \partial_t {W}_t(y)f(x-y)dy$$
				and  $$\displaystyle \int_{\real^n}\partial_t {W}_t(y)dy=0,$$ we have
				\begin{align*}
				\abs{\partial_t {W}_tf(x)} &= \left|\frac{1}{2}\int_{\mathbb{R}^{n}}\partial_t {W}_t(y)(f(x-y)+f(x+y)-2f(x)) dy\right| \\&\le C N_\alpha[f]\,  \int_{\mathbb{R}^n}\frac{e^{-c\frac{|y|^{4/3}}{t^{1/3}}}|y|^\alpha}{t^{{n}/{4}+1}}dy\le C N_\alpha[f]\,   t^{-1+\alpha/4}.
				\end{align*}		
		\end{proof}

Now we can prove Theorem \ref{thm:LipDirivative}.
\begin{proof}[Proof of Theorem \ref{thm:LipDirivative}]
Let us prove the necessity part first. %By Theorem \ref{thm:charclassic}, Remark \ref{rem:all} and \cite[Proposition 8 in Chapter V]{Stein}, we know that $\partial_{x_i}f$ exists and %$\norm{\partial_{x_i}f}_\infty\le C.$
Consider the case $1<\alpha<4$ first.  By Lemma \ref{derivX}(with $j=1,\ m=1$), we have
\begin{equation*}\label{equ:est1}
		\left\| {\partial_t  {W}_t(\partial_{x_i}f)}\right\|_{\infty}=\left\| {\partial_t \partial_{x_i} {W}_tf}\right\|_{\infty}\le C	 \norm{f}_{\dot \Lambda^{\Delta^2}_{\alpha/4}}\, t^{-1+(\alpha-1)/4},
	\end{equation*} for every $ i=1,\cdots,n$.
 Then, we know that $\partial_{x_i}f\in {\dot\Lambda^{\Delta^2}_{(\alpha-1)/4}}$ for every $i=1,\cdots, n.$
When $4n'\le \alpha<4n'+4,\ n'\in \mathbb N,$ then $k=[\alpha/4]+1=n'+1.$ By Lemma \ref{derivX}(with $j=n',\ m=1$), we get
\begin{equation*}%\label{equ:est1}
		\left\| {\partial_t^{n'} {W}_t(\partial_{x_i}f)}\right\|_{\infty}=\left\| {\partial_t^{n'} \partial_{x_i} {W}_tf}\right\|_{\infty}\le C	 \norm{f}_{\dot \Lambda^{\Delta^2}_{\alpha/4}}\, t^{-n'+(\alpha-1)/4},
	\end{equation*} for every $ i=1, \cdots, n$. So, by Proposition \ref{subirk},  $\partial_{x_i}f\in \dot \Lambda^{\Delta^2}_{(\alpha-1)/4}.$

For the converse, assume that $\partial_{x_i}f\in \dot\Lambda^{\Delta^2}_{(\alpha-1)/4}$, $i=1, \cdots, n,$ we want to prove that $f\in \dot \Lambda^{\Delta^2}_{\alpha/4}.$ When $1<\alpha<4,$ \iffalse we have
$$\left\| {\partial_t {W}_t(\partial_{x_i}f)}\right\|_{\infty}\le C	 \norm{\partial_{x_i}f}_{\dot\Lambda^{\Delta^2}_{(\alpha-1)/4}}\, t^{-1+(\alpha-1)/4},$$
for every $ i=1,\cdots,n$. In this case,\fi $[\alpha/4]+1=1.$ By Lemma \ref{derivX}(with $j=0,\ m=3$), we have
$$\left\| \partial^2_{x_{i_1}} \partial_{x_{i_2}}{W}_t(\partial_{x_{i_2}}f)\right\|_{\infty}\le C	 \norm{\partial_{x_{i_2}}f}_{\dot \Lambda^{\Delta^2}_{(\alpha-1)/4}}\, t^{-1+\alpha/4},$$
for any $i_1, i_2\in \{1,\cdots, n\}.$  We would like to prove that
$$\left\| {\partial_t {W}_tf}\right\|_{\infty}\le C	 t^{-1+\alpha/4}.$$
By \eqref{el.4}, we get
\begin{equation}\label{equ:equation}
\partial_t W_tf=\Delta^{2}W_tf=\sum_{i_1=1}^n\sum_{i_2=1}^n\partial^2_{x_{i_1}}\partial^2_{x_{i_2}}W_tf.
\end{equation}
Hence,
\begin{equation*}
\left\| {\partial_t {W}_tf}\right\|_{\infty}\le \sum_{i_1=1}^n\sum_{i_2=1}^n\left\| \partial^2_{x_{i_1}}\partial^2_{x_{i_2}}W_tf\right\|_{\infty}= \sum_{i_1=1}^n\sum_{i_2=1}^n\left\| \partial^2_{x_{i_1}}\partial_{x_{i_2}}W_t(\partial_{x_{i_2}}f)\right\|_{\infty}
\le C	 t^{-1+\alpha/4}.
\end{equation*}
 We consider the case $\alpha\ge 4.$  By the sufficient part of this theorem,
$$\partial^2_{x_{i_1}}\partial^2_{x_{i_2}}f\in \dot \Lambda^{\Delta^2}_{(\alpha-4)/4},$$
for any $i_1, i_2\in \{1,\cdots, n\}.$ Let $k=[\alpha/4]+1.$  Then, %by Proposition \ref{subiraK},
\begin{equation}\label{equ:est2}
\left\| {\partial_t^k {W}_t(\partial^2_{x_{i_1}}\partial^2_{x_{i_2}}f)}\right\|_{\infty}\le C	 t^{-k+(\alpha-4)/4}=C	 t^{-(k+1)+\alpha/4},\end{equation}
for any $i_1, i_2\in \{1,\cdots, n\}.$
Therefore, by \eqref{equ:equation} and  \eqref{equ:est2} we have
\begin{align*}
\left\| {\partial^{k+1}_t {W}_tf}\right\|_{\infty}& =\left\| {\partial^k_t \left( \sum_{i_1=1}^n\sum_{i_2=1}^n\partial^2_{x_{i_1}}\partial^2_{x_{i_2}}\right){W}_tf}\right\|_{\infty} \le \sum_{i_1=1}^n\sum_{i_2=1}^n\left\| {\partial_t^k \left( \partial^2_{x_{i_1}}\partial^2_{x_{i_2}}\right){W}_tf}\right\|_{\infty}\\
&= \sum_{i_1=1}^n\sum_{i_2=1}^n\left\| {\partial_t^k {W}_t\left( \partial^2_{x_{i_1}}\partial^2_{x_{i_2}}f\right)}\right\|_{\infty}
\le C	 t^{-(k+1)+\alpha/4}.
\end{align*}
By Proposition \ref{subirk}, $f\in \dot \Lambda^{\Delta^2}_{\alpha/4}$.
 We complete the proof of the theorem.
\end{proof}

Now, we continue the proof of Theorems \ref{negativepower}--\ref{TRiesz}.
\begin{proof}[Proof of Theorem \ref{negativepower}]
By a similar argument as in the proof of Theorem \ref{Schau}, we can give the proof of Theorem \ref{negativepower}. For completeness, we give a sketch proof in the following.
Let $\ell=[\alpha/4+\beta/4]+1$.  We have
	  	\begin{align*}
	  	&|\partial_t^\ell W_t (({\Delta^2})^{-\beta/4} f)(x)|=\left|\frac{1}{\Gamma(\beta/4)}\int_0^\infty \partial_t^\ell W_t(W_s f)(x) s^{\beta/4}\frac{ds}{s}\right|\\
	  	&\le {C}\int_0^\infty \abs{\partial_w^\ell W_w f(x)\big|_{w=t+s}}s^{\beta/4}\frac{ds}{s}
		\\&\le {C \ S_\alpha[f]\,}\int_0^\infty (t+s)^{-\ell+\alpha/4}s^{\beta/4}\frac{ds}{s}
	  	\\&\stackrel{\frac{s}{t}=u}{\le }C\, S_\alpha[f]\,{t^{\alpha/4+\beta/4-\ell}}\int_0^\infty \frac{u^{\beta/4}}{(1+u)^{\ell-\alpha/4}}\frac{du}{u}
	  	\\&\le C\ S_\alpha[f] \, t^{\alpha/4+\beta/4-\ell}.
	  	\end{align*}		
We end the proof of the theorem.
\end{proof}

\begin{proof}[ {\it Proof of Theorem \ref{Holderestimates}.}]
 	  	Let $\ell=[{\beta/ 4}]+1$ and $m= \left[{(\alpha-\beta)}/{4}\right]+1$. Then, $$m+\ell=\left[{(\alpha-\beta)}/{4}\right]+1+[\beta/4]+1 >\alpha/4-\beta/4+\beta/4=\alpha/4.$$ As $m+\ell\in \mathbb{N}$, we get $m+\ell \ge[\alpha/4]+1.$
	  	
	  	By the semigroup property, we have
	  	\begin{align*}
	  	&\Big|\partial_t^m W_t((\Delta^2)^{\beta/4 } f)(x) \Big|=   \frac1{c_\beta}\Big|\partial_t^m W_t \int_0^\infty
	   (e^{- s\Delta^2}-Id)^{[\beta/4]+1}f(x)\,\frac{ds}{s^{1+\beta/4}} \Big|
\\&= C_\beta\Big| \int_0^\infty \partial_t^m W_t \Big( \underbrace{ \int_0^s\cdots \int_0^s }_{\substack{\ell}}\partial_\nu^\ell W_{\nu} |_{\nu= s_1+\cdots+s_\ell} f(x)ds_1\cdots ds_\ell \Big) \frac{ds}{s^{1+\beta/4}} \Big|\\
 &= 	C_\beta\Big|  \int_0^\infty \Big( \underbrace{\int_0^s\cdots \int_0^s}_{\substack{\ell}} \partial_\nu^{m+\ell}W_{\nu} |_{\nu=t+ s_1+\cdots+s_\ell}f(x) ds_1\cdots ds_\ell \Big) \frac{ds}{s^{1+\beta/4}}\Big|
	  	\\ &\le
	  	C_\beta\,  S_\alpha^\L [f] \int_0^\infty \Big( \underbrace{ \int_0^s\cdots \int_0^s }_{\substack{\ell}}(t+s_1+\cdots +s_\ell)^{-(m+\ell) +\alpha/4} ds_1\cdots ds_\ell \Big) \frac{ds}{s^{1+\beta/4}} \\&=
	  	C_\beta\,  S_\alpha [f] \int_0^t ( \cdots )  \frac{ds}{s^{1+\beta/4}}  + C_\beta\,    S_\alpha [f] \int_t^\infty  (\cdots ) \frac{ds}{s^{1+\beta/4}} \\
&=:C_\beta\,    S_\alpha [f]\,(A +B).
	  	\end{align*}
	   Now we shall estimate $A $ and $B$ separately. By changing variables, we have
		  	\begin{align*}
	  	A  &=
	  	C_\beta t^{-m+\alpha/4} \int_0^t  \underbrace{\int_0^{s/t}\cdots \int_0^{s/t}}_{\substack{\ell}} (1+s_1+\cdots +s_\ell)^{-(m+\ell) +\alpha/4} ds_1\cdots ds_\ell \frac{ds}{s^{1+\beta/4}}  \\&\le
	  	C_\beta\, t^{-m+\alpha/4} \int_0^t   \Big(\frac{s}{t}\Big) ^\ell \frac{ds}{s^{1+\beta/4}} = C_\beta\, t^{-m+\alpha/4-\ell} \int_0^t \frac{ds}{s^{1+\beta/4-\ell}}
	  	= C_\beta\,t^{-m+(\alpha-\beta)/4}.
	  	\end{align*}
	   On the other hand,
	  	\begin{align*}
	  	B  &\le  \int_t^\infty  \sum_{j=0}^\ell \frac{C_j}{(t+js)^{m-\alpha/4}} \frac{ds}{s^{1+\beta/4}} =    \sum_{j=0}^\ell \int_t^\infty \frac{C_j}{(t+js)^{m-\alpha/4}} \frac{ds}{s^{1+\beta/4}} \\& \le   \sum_{j=0}^\ell C_j t^{-m+(\alpha-\beta)/4}
\le C t^{-m+(\alpha-\beta)/4}.
	  	\end{align*}
	  The last inequalities is obtained by observing that $ t\le t+js \le (1+\ell) s$ inside the integrals together with the  discussion  about the sign of $m-\alpha/4$. We complete the proof of this theorem.	  	
	  \end{proof}

	\begin{proof}[{\it Proof of Theorem \ref{TRiesz}.}]
	
	  	Let $0<\alpha\le 1$ and $f\in 	\dot \Lambda_{\alpha/4}^{\Delta^2}$. By Theorem \ref{negativepower} we have  $(\Delta^2)^{-1/4}f \in 	\dot \Lambda_{\frac{\alpha+1}{4}}^{\Delta^2}$.  Therefore, by Theorem \ref{thm:LipDirivative}   we get  $\mathcal{R}_if=\partial_{x_i}(\Delta^2)^{-1/4}f\in 	 \dot \Lambda_{{\alpha}/{4}}^{{\Delta^2}}$.
	  	This gives the first statement of the theorem.

	  	Suppose now $1<\alpha\le 2$ and $f\in 	\dot \Lambda_{\alpha/4}^{\Delta^2}$.  Then, by Theorem \ref{thm:LipDirivative}  we have $\partial_{x_i}f\in 	 \dot \Lambda_{\frac{\alpha-1}{4}}^{{\Delta^2}}$.   Again,  by Theorem \ref{negativepower} we get  $R_if=(\Delta^2)^{-1/4}(\partial_{x_i}f)\in \dot \Lambda_{{\alpha}/{4}}^{{\Delta^2}}$.  	
\end{proof}

At last, we give the proof of 	Theorem \ref{multiplicador}.
	 \begin{proof} [{\it Proof Theorem \ref{multiplicador}.}]	
Firstly, we prove that $m(\Delta^2)$ is bounded on $L^2(\real^n)$.
In this proof we shall use  spectral analysis. The reader can find the details about the spectral analysis in  \cite[Ch.~12~and~13]{Rudin}. Since $\Delta^2$ is a nonnegative, densely defined and self-adjoint operator on $L^2(\mathbb{R}^n)$, there is a unique resolution $E$ of the identity such that
\begin{equation*}\label{phi(L)}
e^{-t\Delta^2}=\int_0^\infty e^{-t\lambda}~dE(\lambda).
\end{equation*}
The above identity  is a shorthand notation that means
$$\langle e^{-t\Delta^2}f,g\rangle_{L^2(\mathbb{R}^n)}=\int_0^\infty e^{-t\lambda}~dE_{f,g}(\lambda),\qquad f, g\in L^2(\mathbb{R}^n),$$
where $dE_{f,g}(\lambda)$ is a regular Borel complex measure of bounded variation concentrated on the spectrum of $\Delta^2$, with $d\abs{E_{f,g}}(0,\infty)\leq\norm{f}_{L^2(\real^n)}\norm{g}_{L^2(\real^n)}$.
Therefore, we write
\begin{eqnarray*}
m(\Delta^2) =\int_0^\infty \int_0^\infty -\partial_s e^{-s\lambda}a(s)d s\ d E(\lambda).
\end{eqnarray*}
And we have
\begin{align*}
\abs{\langle m(\Delta^2)f,g\rangle} &\le \Big|\int_0^\infty \int_0^\infty -\partial_s e^{-s\lambda}a(s)d s\ d E_{f,g}(\lambda)\Big| \\
& \le
\|a\|_{\infty}  \int_0^\infty \int_{0}^{\infty} \abs{\partial_s e^{-s\lambda}}ds ~\abs{dE_{f,g}}(\lambda)\\  &\le C \norm{f}_{L^2(\real^{n})} \norm{g}_{L^2(\real^{n})}.
\end{align*}
This means that $$ \norm{m(\Delta^2)f}_{L^2(\real^n)}\le C \norm{f}_{L^2(\real^n)}.$$

	 Now we want to see that $\|\partial_t^kW_t m(\Delta^2)f\|_\infty\le Ct^{-k+\alpha/4}$. Since $a(s)$ is compactly supported,  Fubini's Theorem together with Lemmas \ref{le2.2}  allows us to interchange integral with derivatives and kernels.  By Proposition \ref{subirk}, we have
\begin{align*}
	 &|\partial_t^kW_t m(\Delta^2)f(x)|=\left|\int_0^\infty \partial_u^{k+1}W_uf(x)\big|_{u=t+s}a(s)ds\right|\\
&\le C\|a\|_\infty S_\alpha[f] \int_0^\infty \frac{1}{(t+s)^{k+1-\alpha/4}}ds\\
	 &=C\|a\|_\infty S_\alpha[f]\, t^{-(k+1)+\alpha/4}\int_0^\infty\frac{ t}{(1+r)^{k+1-\alpha/4}}dr\\
&\le C\|a\|_\infty S_\alpha [f] t^{-k+\alpha/4}.
	  \end{align*}
We complete the proof of the theorem.
\end{proof}

\bigskip

\

\bigskip

{\bf Acknowledgments.}
The author is grateful to Professor J.L. Torrea for his helpful discussion on this project and  the Department of Mathematics at Universidad Aut\'onoma de Madrid for its hospitality during the period of this research.

 \bigskip

 %%%%%%%%%%%%%%%%%%%%%%%%%%%%%%%%%%%%%%%%%%%%%%%%%%%%%%

%%%%%%%%%%%%%%%%%%%%%%%%%%%%%%%%%%%%%%%%%%%%%%%%%%%%%%

\end{document}